\documentclass{amsart}
\usepackage{amsthm}
\usepackage{amsmath,amssymb}
\usepackage{mathrsfs}
\usepackage{bigints}
\usepackage[breaklinks]{hyperref}
\newtheorem{thm}{Theorem}[section]

\newtheorem{lem}[thm]{Lemma}
\newtheorem{prop}[thm]{Proposition}
\theoremstyle{definition}
\newtheorem{defn}[thm]{Definition}
\theoremstyle{remark}
\newtheorem{rem}[thm]{Remark}
\theoremstyle{example}

\numberwithin{equation}{section}

\newcommand{\abs}[1]{\left\vert#1\right\vert}

\newcommand{\norm}[1]{\left\Vert#1\right\Vert}

\begin{document}

\title[The Neumann problem for a class of semilinear fractional equations with critical exponent]
 {The Neumann problem for a class of semilinear fractional equations with critical exponent}
\author[S.\ Gandal, J.\,Tyagi]
{ Somnath Gandal, Jagmohan Tyagi }
\address{Somnath Gandal \hfill\break
 Indian Institute of Technology Gandhinagar \newline
 Palaj, Gnadhinagar Gujarat India-382355.}
\email{gandal$\_$somnath@iitgn.ac.in}
\address{Jagmohan\,Tyagi \hfill\break
 Indian Institute of Technology Gandhinagar \newline
 Palaj, Gandhinagar Gujarat, India-382355.}
 \email{jtyagi@iitgn.ac.in, jagmohan.iitk@gmail.com}
\date{\today}
\thanks{Submitted \today.  Published-----.}

\begin{abstract}
We establish the existence of solutions to the following semilinear Neumann problem for fractional Laplacian and critical exponent:
\begin{align*}\left\{\begin{array}{l l} { (-\Delta)^{s}u+ \lambda u= \abs{u}^{p-1}u } & \text{in $ \Omega,$  } \\ 
\hspace{0.8cm} { \mathcal{N}_{s}u(x)=0 } & \text{in $ \mathbb{R}^{n}\setminus \overline{\Omega},$} \\
\hspace{1.6cm} {u \geq 0}& \text{in $\Omega,$} \end{array} \right.\end{align*}
where $\lambda > 0$ is a constant and $\Omega \subset \mathbb{R}^{n}$ is a bounded 
domain with smooth boundary. Here, $p=\frac{n+2s}{n-2s}$ is a critical exponent, $n > \max\left\{4s, \frac{8s+2}{3}\right\},$  $s\in(0, 1).$
Due to the critical exponent in the problem, the corresponding functional $J_{\lambda}$ does not satisfy the Palais-Smale (PS)-condition and therefore one cannot use standard variational methods to find the critical points of $J_{\lambda}.$ We overcome such difficulties by establishing a bound for Rayleigh quotient and with the aid of nonlocal version of the Cherrier's optimal Sobolev inequality in bounded domains.  We also show the uniqueness of these solutions in small domains.
\end{abstract}

\maketitle

\tableofcontents

\section{Introduction}


We establish the existence of non-constant solutions to the following semilinear Neumann problem with fractional Laplace operator 
$(-\Delta)^{s},~s\in (0,1)$:  
\begin{align}\label{eqn:P1}\left\{\begin{array}{l l} { (-\Delta)^{s}u+ \lambda u= \abs{u}^{p-1}u } & \text{in $\Omega,$  } \\ 
		\hspace{0.9cm} { \mathcal{N}_{s}u(x)=0 } & \text{in $\mathbb{R}^{n}\setminus \overline{\Omega},$} \\
		\hspace{1.8cm} {u \geq 0} & \text{in $\Omega,$} \end{array} \right.\end{align}
where $\Omega \subset \mathbb{R}^{n}$ is a $C^{1},$ smooth bounded domain, $\lambda >0$ is a constant and $p = \frac{n+2s}{n-2s}, n> \max\left\{4s, \frac{8s+2}{3}\right\}.$
For smooth functions, the fractional Laplace operator $(-\Delta)^{s}$ is defined 
as follows:
\begin{equation}
(-\Delta)^{s}u(x):=c_{n,s}P.V.\int_{\mathbb{R}^{n}} \frac{u(x)-u(y)}{{\lvert x-y \rvert}^{n+2s}}dy, \,\,\, s\in (0,1).
\end{equation}
Here, by P.V., we mean, ``in the principal value sense" and $c_{n,s}$ is a normalizing constant given by
\begin{equation}
\label{eqn:ncon}
c_{n,s}= \biggl(\int_{\mathbb{R}^{n}}\frac{1-cos(\theta_{1})}{{\lvert \theta \rvert}^{n+2s}}d\theta \biggr )^{-1},
\end{equation}
see for example \cite{B4,E1}. By the operator $\mathcal{N}_{s}w,$ we mean nonlocal 
normal derivative defined for smooth functions by
\begin{equation}
	\mathcal{N}_{s}w(x):= c_{n,s}\int_{\Omega} \frac{w(x)-w(y)}{{\lvert x-y \rvert}^{n+2s}}dy, 
	\,\,\,\,\, x\in \mathbb{R}^{n} \setminus \overline{\Omega},
\end{equation} 
where $c_{n,s}$  is the same normalizing constant given in (\ref{eqn:ncon}).
\begin{rem}
	The nonlocal normal derivative $\mathcal{N}_{s}w$ defined above approaches the classical one $\partial_{\nu}w$ 
	when $s\rightarrow 1$ (see, Proposition 5.1 \cite{S1}). 
\end{rem}
In the  recent years, there has been a growing interest to establish the existence, asymptotic behaviour and other 
related qualitative 
questions to nonlocal problems which involve $(-\Delta)^{s}.$
See for example \cite{B5,B1,B4,S2,J1,G1,X1,X2} and references therein. The study on these problems with nonlocal 
diffusion operators is motivated by its applications in image processing \cite{G1}, fluid dynamics \cite{N1,A7,R1},
finances \cite{P1,W1} and many others. For more about Neumann problems which include the fractional Laplacian $(-\Delta)^{s}$ as well as other 
$s$-nonlocal operators, we refer to \cite{Barl,Cor,S1,Gru,Mon} and the references therein.

Problem (\ref{eqn:P1}) is a fractional counterpart of the following classical Neumann problem: 
\begin{equation}\label{eqn:Pi}
\left\{\begin{array}{l l} { -\Delta u+ \lambda u= \abs{u}^{p-1}u } & \text{in $\Omega,$  } \\ 
\hspace{1.6cm} { \frac{\partial u}{\partial \nu}=0 } & \text{on $\partial \Omega,$} \\
\hspace{1.8cm} {u>0} & \text{in $\Omega,$} \end{array} \right.
\end{equation}
where $\lambda >0$ is a constant and $\Omega \subset \mathbb{R}^{n}$ is bounded domain with smooth boundary. There has been a good amount of works where the authors have investigated \eqref{eqn:Pi} in different contexts. 
We mention only closely related papers to bring up earlier developments.
The subcritical case, i.e., $1<p<\frac{n+2}{n-2}, n\geq 3$ of (\ref{eqn:Pi}) has been deeply studied in the papers of Lin et al. \cite{C1} and C. S. Lin and W. M. Ni \cite{C2}. They have shown that (\ref{eqn:Pi}) admits a
nonconstant solution for $\lambda$ sufficiently large and does not have any nonconstant solution for $\lambda$ small.
The situation in the critical case, i.e., $p=\frac{n+2}{n-2}$ is quite different than that of subcritical case. The nonconstant solutions of (\ref{eqn:Pi}) correspond to the critical points of the energy functional 
$$ F_{\lambda}(u)= \frac{1}{2}\int_{\Omega} \abs{\nabla u}^{2}dx +\frac{\lambda}{2} 
\int_{\Omega}u^{2}dx-\frac{1}{p+1}\int_{\Omega}\abs{u}^{p+1}dx, \hspace{0.3cm}u\in H^{1}(\Omega).$$
Due to lack of compactness for the embedding $H^{1}(\Omega) \hookrightarrow L^{\frac{2n}{n-2}}(\Omega),$ 
the energy functional $F_{\lambda}$ does not satisfy the Palais-Smale (PS) condition. Therefore, one cannot use the 
standard variational methods to find the critical points of energy functional. 
It is crucial to mention that the geometry of boundary plays an important role to establish the existence and 
non-existence of solutions in the critical case.  Adimurthi and G. Mancini \cite{A2} and X. J. Wang \cite{X3} have shown the 
existence of nonconstant solutions of (\ref{eqn:Pi}) with critical exponent for $\lambda$ sufficiently large in general 
domains $\Omega.$ 
When $\Omega$ is a ball and $\lambda$ is very small, Adimurthi and S. L. Yadava \cite{A6} and Budd et al. \cite{C3} have proved the existence of nonconstant solution in lower dimensions $n=4,5,6.$ In those works, authors have considered the 
equations with critical exponents. 

We remark that in elliptic problems, where the associated energy functional does not satisfy (PS)-condition, Cherrier's optimal Sobolev inequlity plays a very crucial role, see \cite{A1}. One may refer to Theorem 2.30 \cite{Aub} and \cite{Ch} for the details about the Cherrier's optimal Sobolev inequality.

Recently, E. Cinti and F. Colasuonno \cite{Cin} considered nonlocal problems with Neumann boundary condition and established the existence of nonnegative radial solutions to the following problem:
\begin{align}\left\{\begin{array}{l l} { (-\Delta)^{s}u+ u= h(u) } & \text{in $\Omega,$  } \\ 	\hspace{0.9cm} { \mathcal{N}_{s}u(x)=0 } & \text{in $\mathbb{R}^{n}\setminus \overline{\Omega},$}  \end{array} \right.\end{align}
where $s>\frac{1}{2},$ $\Omega$ is either a ball or annulus and $h$ is a superlinear nonlinearity.
In \cite{R2,R3}, R. Servadei and E. Valdinoci studied the Brezis-Nirenberg problem with Dirichlet boundary in  
nonlocal case for 
subcritical as well as critical exponent. More precisely, they considered the nonlinearity of the
form $ \lambda u + u^{p-1},$ with $\lambda > 0$ and $1< p \leq\frac{n+2s}{n-2s}.$ Bhakta et al. \cite{B6,B7} studied the 
semilinear problem for fractional Laplacian with nonlocal Dirichlet datum with nonlinearity involving critical and 
supercritical growth. We refer to \cite{mug}, where D. Mugnai and E. Proietti Lippi established the existence of nontrivial solutions 
for Neumann fractional $p$-Laplacian using the notion of linking over cones. We refer to \cite{nsu} 
for the existence of solutions of a semilinear problem with the spectral Neumann fractional Laplacian 
and a critical nonlinearity.\\
\indent It is easy to see that $u_{0}=0$ and $u_{\lambda}= \lambda^{\frac{1}{p-1}}= \lambda^{\frac{(n-2s)}{4s}}$ are the only constant
solutions of (\ref{eqn:P1}). Therefore, motivated by the local case, we are interested in finding nonconstant solutions of
(\ref{eqn:P1}). Barrios et al. \cite{B1} have shown that there is a nonconstant solution to (\ref{eqn:P1}) in the 
subcritical case $ 1<p<\frac{n+2s}{n-2s}.$ The nonconstant solutions of this problem corresponds to the critical points of 
the energy functional $J_{\lambda},$ where $J_{\lambda}$ is given by
$$ J_{\lambda}(u):=\frac{1}{2} \Big[ \frac{c_{n,s}}{2} \int_{T(\Omega)} \frac{\abs{u(x)-u(y)}^{2}}{\abs{x-y}^{n+2s}}dxdy + 
\lambda \int_{\Omega}u^{2}dx \Big] -\frac{1}{p+1}\int_{\Omega}{\lvert u \rvert}^{p+1}dx, \,\,\,\, u \in H_{\Omega}^{s}.$$
In above equation $T(\Omega)= \mathbb{R}^{2n}\setminus (\mathbb{R}^{n}\setminus \Omega)^{2}$ and the space $H_{\Omega}^{s}$ is
defined further in (\ref{space}).
They have used the Mountain-Pass Lemma by Ambrosetti-Rabinowitz \cite{A4} to find the critical points of this 
functional. Since $H^{s}_{\Omega}$ is not compactly embedded in  $L^{\frac{2n}{n-2s}}(\Omega),$ 
the classical variational methods fail in the critical case (i.e., $p=\frac{n+2s}{n-2s})$ 
to find the critical points of $J_{\lambda}.$ In such a challenging situation, it is significant to establish the following inequality for suitable values of $\lambda$:
\begin{align} \label{optimal}
\inf_{u \in H^{s}_{\Omega}, \norm{u}_{L^{p+1}(\Omega)}=1} \left\{\frac{c_{n,s}}{2} \int_{T(\Omega)} \frac{\abs{u(x)-u(y)}^{2}}{\abs{x-y}^{n+2s}}dxdy + \lambda \int_{\Omega}u^2 dx \right\} < \frac{S}{2^{\frac{2s}{n}}},
\end{align}
where $S$ is the best constant for the fractional Sobolev embedding $$ H^{s}(\mathbb{R}^n) \hookrightarrow L^{\frac{2n}{n-2s}}(\mathbb{R}^{n}).$$ For the classical case, we refer to some well known works of T. Aubin \cite{Au}, H. Brezis and L. Nirenberg \cite{B3}, H. Brezis \cite{Bre} and Adimurthi et al. \cite{A2, A1}. In order to establish the above inequality (\ref{optimal}), we evaluate the ratio
\begin{align}
	K_{\lambda}(V_{\epsilon})=\frac{\frac{c_{n,s}}{2} \int_{T(\Omega)}\frac{{\lvert V_{\epsilon}(x)-V_{\epsilon}(y)\rvert}^{2}}
		{{\lvert x-y \rvert}^{n+2s}}dxdy+ \lambda\int_{\Omega} V_{\epsilon}^{2}dx}{\biggl(\int_{\Omega} 
		{\lvert V_{\epsilon} \rvert}^{p+1}dx\biggr)^{\frac{2}{p+1}}},
\end{align}
where $$V_{\epsilon}(x)= \frac{\psi(x)}{(\epsilon + \abs{x}^{2})^{\frac{(n-2s)}{2}}}, \, \epsilon >0$$ and 
$\psi$ is a cut-off function. The functions $${(\epsilon + \abs{x}^{2})^{\frac{-(n-2s)}{2}}}$$ are extremals for the fractional Sobolev 
embedding in $\mathbb{R}^{n}$. 
Next, we obtain the following nonlocal analog of Cherrier's optimal Sobolev inequality. 
\begin{thm} (Nonlocal version of Cherrier's optimal Sobolev inequality)\label{best4}
Let $\Omega \subset \mathbb{R}^n$ be a bounded domain of class $C^{1}.$ Then for every $\epsilon>0,$ there exists a constant $A(\epsilon)>0$ such that for any $u \in H^{s}_{\Omega},$ we have
	\begin{align}\label{best4.01}
		\left(\int_{\Omega} \abs{u}^{p+1} dx \right)^{\frac{2}{p+1}} \leq \left(\frac{2^{\frac{2s}{n}}}{S}+\epsilon \right) \frac{c_{n,s}}{2} \int_{T(\Omega)} \frac{\abs{u(x)-u(y)}^2}{\abs{x-y}^{n+2s}} dxdy + A(\epsilon) \int_{\Omega} u^2 dx.
	\end{align} 
\end{thm}
The inequalities (\ref{optimal}) and (\ref{best4.01}) are instrumental to prove the following existence result of this paper.
\begin{thm} \label{2}
	Let $\Omega \subset \mathbb{R}^{n} $ be a bounded domain of class $C^{1}.$ Let $p=\frac{n+2s}{n-2s}$, $n > \max\left\{4s, \frac{8s+2}{3}\right\}$ 
	and $s\in (0,1).$ Then there exists a constant $\lambda^{*}>0$ such that for $\lambda> \lambda^{*},$ the problem
	\begin{equation}\label{eqn:P2}
		\left\{\begin{array}{l l} { (-\Delta)^{s}u+ \lambda u= \abs{u}^{p-1}u } & \text{in $\Omega,$  } \\ 
			\hspace{0.9cm} { \mathcal{N}_{s}u(x)=0 } & \text{in $\mathbb{R}^{n}\setminus \overline{\Omega},$} \\
			\hspace{1.8cm} {u\geq 0} &\text{in $\Omega$,} \end{array} \right.
	\end{equation}
	admits a non-constant solution $v_{0}$ such that $J_{\lambda}(v_{0})<\frac{s S^{\frac{n}{2s}}}{2n}.$ 
\end{thm}

\indent There has also been a good amount of interest to study the asymptotic behaviour of minimal energy solutions to 
semilinear Neumann problems. W. M. Ni and I. Takagi \cite{W2,W3} have shown that the minimal energy solutions $v_{\lambda}$ of 
(\ref{eqn:Pi}) attains the maximum at unique point on the boundary of $\Omega$ for $\lambda$ large. 
Even in the critical case, the same has been proved by Adimurthi et al. \cite{A8}. There has also been some other type of results on the asymptotic behaviour. For this,
we consider the following family of domains
$$ \Omega_{\eta}=\bigl \{\eta x : x\in \Omega \bigr \} $$
to the asymptotic behaviour of the solutions. For contracting domains, i.e, $\eta \rightarrow 0$ and expanding domains,
i.e., $\eta \rightarrow \infty,$
the detailed analysis of the behaviour of best Sobolev constants and extremals, occurring  due to classical Sobolev 
embedding can be seen in \cite{C4,J1,J3} and the references therein. Let $S_{q}(\Omega)$ be the best constant for
the Sobolev trace embedding $ W^{1,p}(\Omega) \hookrightarrow L^{q}(\partial \Omega),$
where $n \geq 2, 1<p<n,$ and $1<q\leq \frac{p(n-1)}{n-p}.$ 
In particular, when $n \geq 3,~ p=2$ and
$ 2<q< \frac{2(n-1)}{n-2},$ 
del Pino et al. \cite{C4} studied the asymptotic behaviour of best constant $S_{q}(\Omega)$ and associated family of 
extremals in expanding domains. They have shown that $S_{q}(\Omega_{\lambda})$ approaches to $S_{q}(\mathbb{R}_{+}^{n})$ 
and the 
extremals develop a peak near the point, where the curvature of the boundary attains a maximum. In more general case, 
when $1<p<n, 1<q < \frac{p(n-1)}{n-p},$  Fern\'{a}ndez Bonder et al.\,\cite{J3} discussed the asymptotic behaviour 
of $S_{q}(\Omega)$ 
in expanding and contracting domains. In that paper, they have proved that the behaviour of $S_{q}(\Omega)$ depends on 
$p$ and $q$. Also, the extremals converge to a constant in contracting domains and concentrates at the boundary in 
expanding domains.\\
\indent Let us define 
\begin{align}
Y_{p}(\Omega):=\inf_{v\in H^{s}(\Omega)\setminus \{0\} } \frac{\norm{v}_{H^{s}(\Omega)}^{2}}{\norm{v}_{L^{p+1}(\Omega)}^{2}}.
\end{align}
For smooth bounded domains $\Omega \subset \mathbb{R}^{n}$, we have that $H^{s}(\Omega)$  is compactly embedded in 
$L^{p+1}(\Omega)$ for $0\leq p<\frac{n+2s}{n-2s}$. Therefore, we have the existence of extremal for above equation 
in subcritical case $1<p<\frac{n+2s}{n-2s}.$ Any extremal for $Y_{p}$ is a weak solution to following 
problem $$ (-\Delta)^{s}u+u=\lambda \abs{u}^{p-1}u,$$ where $\lambda$ is some constant.
Very recently,  Fern\'{a}ndez Bonder et al. \cite{J2} described the asymptotic behaviour of the best constants $ Y_{p}(\Omega_{\eta})$ and 
corresponding extremals as $\eta$ goes to zero in the \textit{subcritical} case. More precisely, let $u_{\eta}$ be an extremal for $Y_{p}(\Omega_{\eta}).$ Then it has been shown that any
rescaled extremals $$\bar{u}_{{\eta}}=u_{\eta}(\eta x)$$ when normalized as $$\norm{\bar{u}_{{\eta}}}_{L^{p+1}{(\Omega)}}=1 $$ 
converge strongly to $\abs{\Omega}^{\frac{-1}{p+1}}$ in $H^{s}(\Omega).$  Using this  asymptotic behaviour, they showed the uniqueness of minimal energy solutions for contracted domains, see \cite{J2}.

Motivated by the above work, we show the uniqueness of minimal energy solutions of (\ref{eqn:P1}). We prove the uniqueness of minimal energy 
solutions using implicit function theorem and following the similar approach as in \cite{J2} for \textit{subcritical} case. More specifically, we have 
\begin{thm}\label{minimizer4.6}
Let $\Omega \subset \mathbb{R}^{n},~ p,~ n,~s$ and $\lambda^{*}$ are same as in the Theorem \ref{2}. Let $\lambda > \lambda^{*},$ then there exists 
a $\delta>0$ such that $X_{\lambda}(\Omega_{\eta})$ has a unique minimizer for $0<\eta<\delta$, where $X_{\lambda}(\Omega_{\eta})$
is defined in (\ref{X}).
\end{thm}
We organize this paper as follows. Section 2 contains useful results, which are used in the paper. Section 3 deals with the proof of Theorem \ref{best4}. In Section 4, we prove Theorem \ref{2}. We give the proof of 
Theorem \ref{minimizer4.6} 
for thin domains in Appendix.
\section{Preliminaries}
Let us recall the important results which are used in this paper.
\begin{thm}\label{fembed}(Fractional Sobolev Embedding \cite{E1})  Let $n>2s$ and $p=\frac{n+2s}{n-2s}$ be 
the fractional critical exponent. Then we have the following inclusions:
\begin{enumerate}
\item for any measurable function $u\in C_{0}(\mathbb{R}^{n}),$ we have for $q\in [0, p]:$ $$ 
\norm{u}_{L^{q+1}(\mathbb{R}^{n})}^{2} \leq B(n,s)\int_{\mathbb{R}^{n}}\int_{\mathbb{R}^{n}}
\frac{\abs{u(x)-u(y)}^{2}}{\abs{x-y}^{n+2s}}dxdy $$
for some constant B depending upon n and s. That means $H^{s}(\mathbb{R}^{n})$ is continuously embedded 
in $L^{q+1}(\mathbb{R}^{n}).$
\item Let $\Omega \subset \mathbb{R}^{n}$ be a bounded extension domain for $H^{s}(\Omega).$
Then the space $H^{s}(\Omega)$ is continuously embedded in $L^{q}(\Omega)$ for any $q\in [0, p],$ i.e, 
$$ \norm{u}_{L^{q+1}(\Omega)}^{2} \leq B(n,s)\norm{u}_{H^{s}(\Omega)}^{2}$$
for some constant B depending upon $n,s$ and $\Omega.$
Further, the above embedding is compact for any $q\in [0, p).$
\end{enumerate}
\end{thm}
\noindent Next theorem accords the best constant for the embedding 
$H^{s}(\mathbb{R}^{n}) \hookrightarrow L^{\frac{2n}{n-2s}}(\mathbb{R}^{n}).$
\begin{thm} \label{Best}(Theorem 1.1, \cite{A3})
Let $n>2s$ and $p+1= \frac{2n}{n-2s}.$ Then
\begin{align} \label{Best3}
S \biggl ( \int_{\mathbb{R}^{n}} \abs {u}^{p+1} \biggr )^{\frac{2}{p+1}} \leq \frac{c_{n,s}}{2}
\int_{\mathbb{R}^{n}}\int_{\mathbb{R}^{n}} \frac{\abs{u(x)-u(y)}^{2}}{\abs{x-y}^{n+2s}}dxdy, \,\,\,\,\,\,\forall~ u\in H^{s}(\mathbb{R}^{n}),
\end{align}
where $c_{n,s}$ is the normalizing constant defined in (\ref{eqn:ncon}) and $ S$ is the sharp constant in 
the above inequality. Precisely, 
$$ S= 2^{2s}\pi^{s} \frac{\Gamma (\frac{n+2s}{2})}{\Gamma (\frac{n-2s}{2})} \biggl [\frac{\Gamma(\frac{n}{2})}{\Gamma(n)} 
\biggr ]^{\frac{2s}{n}}. $$
We have equality in (\ref{Best3}) iff  $u(x)=\frac {c}{\left(\epsilon + \abs{x-z_{0}}^{2}\right)^{\frac{n-2s}{2}}},~~x\in \mathbb{R}^{n} $,
where $c \in \mathbb{R},~\epsilon >0 ~ and ~z_{0} \in \mathbb{R}^{n}$ are fixed constants.
\end{thm}

The next lemma is an improved version of Fatou's lemma, which was used by Adimurthi and Mancini \cite{A2}
and Adimurthi and Yadava \cite{A1}.
\begin{lem} \label{BL}(Brezis-Lieb \cite{B2}) Let $(\Omega, \mu )$ be a measurable space. 
If $ \{u_{k}\} $ is a bounded sequence in 
$ L^{q}(\Omega, \mu)$, $ q \in (1, \infty) $ and $u_{k}\rightarrow u $ a.e., then
$$ \int_{\Omega} {\abs u_{k}^{q}}d\mu- \int_{\Omega}{\abs u}^{q}d\mu-\int_{\Omega} {\abs {u_{k}-u}^{q}}d \mu \rightarrow 0. $$
\end{lem}

\noindent When $q=2,$ the conclusion of Brezis-Lieb lemma holds even if convergence a.e. 
is not assumed. That means if $u_{k} \rightharpoonup u, $ weakly in $L^{2}(\Omega),$ then
\begin{align}
{\norm {u_{k}}}_{L^{2}(\Omega)}^{2}& =  {\norm {u_{k}-u}}_{L^{2}(\Omega)}^{2}+ {\norm {u}}_{L^{2}(\Omega)}^{2}-2(u_{k}-u, u) 
\nonumber \\
& =  {\norm {u_{k}-u}}_{L^{2}(\Omega)}^{2}+ {\norm {u}}_{L^{2}(\Omega)}^{2} +o(1)\nonumber.
\end{align}
Let $$T(\Omega):=\mathbb{R}^{2n}\setminus (\mathbb{R}^{n}\setminus \Omega)^{2}.$$ Define
\begin{equation}\label{space}
H_{\Omega}^{s}:=\left\{ u: \mathbb{R}^{n}\rightarrow \mathbb{R}~\text{measurable}: {\norm {u}}_{H_{\Omega}^{s}}< \infty \right \}
\end{equation}
which is equipped with the norm
\begin{equation}
\norm{u}_{H_{\Omega}^{s}}^{2}:=\biggl(\norm{u}_{L^{2}
(\Omega)}^{2}+\int_{T(\Omega)}\frac{{\lvert u(x)-u(y)\rvert}^{2}}{{\lvert x-y \rvert}^{n+2s}}dxdy \biggr).
\end{equation}
\begin{rem}
$H_{\Omega}^{s}$ is a Hilbert space (see \cite{S1}, Proposition 3.1).
\end{rem}
\begin{rem}\label{rem1}
Let $u\in H_{\Omega}^{s}$, then $$\int_{\Omega}\int_{\Omega}\frac{{\lvert u(x)-u(y)\rvert}^{2}}{{\lvert x-y
\rvert}^{n+2s}}dxdy \leq \int_{T(\Omega)}\frac{{\lvert u(x)-u(y)\rvert}^{2}}{{\lvert x-y \rvert}^{n+2s}}dxdy <\infty.$$
This implies that restriction of $u$ to $\Omega$ is in $H^{s}(\Omega)$ and so we have a continuous inclusion $ H_{\Omega}^{s} \hookrightarrow H^{s}(\Omega)$.
\end{rem}
\noindent In fact, we have the following embedding:
\begin{prop}(Proposition 2.4 \cite{Cin}) \label{compact}
Let $\Omega \subset \mathbb{R}^{n}$ be a bounded domain of class $C^{1}$ and $p= \frac{n+2s}{n-2s}, n >2s,$ then  $H^{s}_{\Omega}$ is compactly embedded in $L^{q}(\Omega)$ for any $q \in [1, p+1).$
\end{prop}
\noindent Now, let us recall the integration by parts formula \cite{S1}.
\begin{lem} \label{byparts}
For bounded $C^{2}$ functions $v, w$ in $\mathbb{R}^{n},$ we have
$$ \int_{\Omega}w(-\Delta)^{s}v dx = \frac{c_{n,s}}{2}\int_{T(\Omega)} \frac{(v(x)-v(y))(w(x)-w(y))}{\abs {x-y}^{n+2s}} dxdy - 
\int_{\mathbb{R}^{n}\setminus \Omega} \mathcal{N}_{s}(v)w dx, $$
where $c_{n,s}$ is a normalizing constant in (\ref{eqn:ncon}). 
\end{lem}

\noindent The integration by parts formula leads to the following weak formulation of Neumann problem.
\begin{defn} \label{weak}
 Let $ h\in L^{2}(\Omega)$ and $u \in H_{\Omega}^{s}.$ We say that $u$ is a weak solution of 
\[ \begin{cases} 
(-\Delta)^{s}u+\lambda u = h & \text{in $\Omega,$} \\
\hspace{1.4cm}\mathcal{N}_{s}u=0 &  \text{in $\mathbb{R}^{n}\setminus \overline{\Omega},$}
\end{cases} \]
whenever $$ \frac{c_{n,s}}{2}\int_{T(\Omega)}\frac{(u(x)-u(y))(w(x)-w(y))}{\abs{x-y}^{n+2s}}dxdy + \lambda \int_{\Omega}uw dx
=\int_{\Omega} hw dx \,\,\,\,\text{holds}\,\,\,\forall\, w \in H_{\Omega}^{s}.$$
\end{defn}

\section{Cherrier's Optimal Sobolev Inequality}
Let $s\in(0, 1),\, n>2s.$ Throughout this section, we fix the exponent $p=\frac{n+2s}{n-2s}$ 
and $\Omega \subset \mathbb{R}^n$ be a bounded domain of class $C^{1}$. Let $T(\Omega):=\mathbb{R}^{2n}\setminus (\mathbb{R}^{n}\setminus \Omega)^{2}$ be a cross shaped set on $\Omega$ and $S$ is the best Sobolev constant for the embedding $H^{s}(\mathbb{R}^n) \hookrightarrow L^{p+1}(\mathbb{R}^n)$. This section deals with the generalization of Cherrier's optimal Sobolev inequality in nonlocal case, which plays an important role in our existence theorem. Let us define
\begin{align*}
	E:= \left\{x=(x_{1},x_{2}, \dots , x_{n}) \in \mathbb{R}^{n} \mid x_{n}>0 \right\}
\end{align*}
and $$D := \overline{E} \cap B_{1} .$$
Following are the important lemmas, which are useful to obtain the optimal fractional Sobolev inequality on bounded domains. These lemmas are borrowed from \cite{E1} with some modifications.
\begin{lem} \label{Best1.1} (Lemma 5.1 \cite{E1})
	Let $h : \mathbb{R}^{n}  \rightarrow \mathbb{R}$ be a measurable function which is $C^{1}$ on $B_{r}$ for some $r>0$ and has support in  $B_{r}.$ Then $h$ satisfies:
	\begin{align}
		\left(\int_{B_{r}} \abs{h}^{p+1} dx \right)^{\frac{2}{p+1}} \leq \frac{1}{S}\left( \frac{c_{n,s}}{2} \int_{T(B_{r}(0))} \frac{\abs{h(x)-h(y)}^2}{\abs{x-y}^{n+2s}} dxdy \right).
	\end{align} 
\end{lem}
\begin{proof}
	Let $K= supp(h) \subset B_{r}$ be a compact set. Define
	\[ \overline{h}(x):= \begin{cases}
			h(x), & \text{$ x \in B_{r} $}, \\
			0, &  \text{ $x \in \mathcal{C}B_{r}.$} \end{cases} \]  
Then it is easy to see that $\overline{h} \in H^{s}(\mathbb{R}^{n}).$ Now, by Theorem \ref{Best}, we get
\begin{align}
	\left(\int_{B_{r}} \abs{h}^{p+1} dx \right)^{\frac{2}{p+1}} = \left(\int_{\mathbb{R}^n} \abs{\overline{h}}^{p+1} dx \right)^{\frac{2}{p+1}}& \leq \frac{1}{S}\left( \frac{c_{n,s}}{2} \int_{\mathbb{R}^n} \int_{\mathbb{R}^n} \frac{\abs{\overline{h}(x)-\overline{h}(y)}^2}{\abs{x-y}^{n+2s}} dxdy \right) \nonumber \\
	& \leq \frac{1}{S}\left( \frac{c_{n,s}}{2} \int_{T(B_{r}(0))} \frac{\abs{h(x)-h(y)}^2}{\abs{x-y}^{n+2s}} dxdy \right).
\end{align}
\end{proof}
\begin{lem} (Lemma 5.2 \cite{E1}) \label{Best1.2}
	Let $h : \mathbb{R}^{n}  \rightarrow \mathbb{R}$ be a measurable function which is $C^{1}$ on $E$ and has support in $D.$ Then $h$ satisfies:
	\begin{align}
		\left(\int_{E} \abs{h}^{p+1} dx \right)^{\frac{2}{p+1}} \leq \frac{2^{\frac{2s}{n}}}{S} \left(\frac{c_{n,s}}{2} \int_{T(E)} \frac{\abs{h(x)-h(y)}^2}{\abs{x-y}^{n+2s}} dxdy \right).
	\end{align}  
\end{lem}
\begin{proof}
	Define 
	\[ \widetilde{h}(x):= \begin{cases}
		h(x), & \text{if $ x \in E $}, \\
		h(\widetilde{x}), &  \text{if  $x \in \mathcal{C}E,$} \end{cases} \]  
	where $\widetilde{x}= (x_{1},x_{2}, \dots, -x_{n}).$
	It is easy to check that  $\widetilde{h} $ is a $C^{1}$ function with compact support in $\mathbb{R}^n.$ Thus $\widetilde{h} \in H^{s}(\mathbb{R}^n).$ Using fractional Sobolev embedding theorem, we get
\begin{align}
		\norm{\widetilde{h}}_{L^{p+1}(\mathbb{R}^n)}^2 \leq \frac{1}{S} \left[\frac{c_{n,s}}{2} \int_{\mathbb{R}^n} \int_{\mathbb{R}^n}  \frac{\abs{\widetilde{h}(x)-\widetilde{h}(y)}^2}{\abs{x-y}^{n+2s}} dxdy \right].
	\end{align} 
Since
	\begin{align*}
		\int_{\mathbb{R}^n} \abs{\widetilde{h}}^{p+1}=2\int_{E} \abs{h}^{p+1}
	\end{align*}
	and 
	\begin{align*}
		\int_{\mathbb{R}^n} \int_{\mathbb{R}^n}  \frac{\abs{\widetilde{h}(x)-\widetilde{h}(y)}^2}{\abs{x-y}^{n+2s}} dxdy \leq 2 \int_{T(E)} \frac{\abs{h(x)-h(y)}^2}{\abs{x-y}^{n+2s}} dxdy, 
	\end{align*}
so the lemma follows.
\end{proof}
\noindent We use the partition of unity to prove Theorem \ref{best4}.\\

\noindent\textbf{Proof of Theorem \ref{best4}:}
	Let $\left\{ h_{i} \right\}_{i \in I}$ be a $C^{\infty} $ partition of unity subordinate to the finite open covering $\left\{ \Omega_{i} \right\}_{i \in I}$ of $\overline{\Omega},$ each $\Omega_{i}$ being homeomorphic to a ball of $\mathbb{R}^n$ or to a half ball $D$ as defined above. Let $\abs{I}=k$ for some integer $k>0.$ We prove the theorem for $u \in H^{s}_{\Omega} \cap C^{\infty}(\overline{\Omega}).$ The case where $u \in H^{s}_{\Omega}$ can be proved by approximation. By Lemmas \ref{Best1.1} and \ref{Best1.2}, we have
{\allowdisplaybreaks	\begin{align*}
		\sum_{i \in I}^{} \norm{u^2 h_{i}}_{L^{\frac{p+1}{2}}(\Omega_{i})}  & = \sum_{i \in I}^{} \norm{uh_{i}^{\frac{1}{2}}}_{L^{p+1}(\Omega_{i})}^2 \\
		& \leq \frac{2^{\frac{2s}{n}}}{S} \sum_{i \in I}^{}  \frac{c_{n,s}}{2} \int_{T(\Omega_i)} \frac{\abs{u(x)h_{i}^{\frac{1}{2}}(x)-u(y)h_{i}^{\frac{1}{2}}(y)}^2}{\abs{x-y}^{n+2s}} dxdy \\
		& \leq \frac{2^{\frac{2s}{n}}}{S} \sum_{i \in I}^{}  \frac{c_{n,s}}{2} \int_{T(\Omega)} \frac{\abs{u(x)h_{i}^{\frac{1}{2}}(x)-u(y)h_{i}^{\frac{1}{2}}(y)}^2}{\abs{x-y}^{n+2s}} dxdy \\
		&  \leq \frac{2^{\frac{2s}{n}}}{S} \left(\frac{c_{n,s}}{2} \int_{T(\Omega)} \sum_{i \in I}^{}  \left( \frac{\abs{(u(x)-u(y))h_{i}^{\frac{1}{2}}(x) + u(y) (h_{i}^{\frac{1}{2}}(x)-h_{i}^{\frac{1}{2}}(y)}^2}{\abs{x-y}^{n+2s}} dxdy\right)\right)\\
		& \leq \frac{2^{\frac{2s}{n}}}{S} \Biggl(\frac{c_{n,s}}{2} \int_{T(\Omega)} \sum_{i \in I}^{}  \Biggl( \frac{\abs{(u(x)-u(y))}^2\abs{h_{i}(x)}}{\abs{x-y}^{n+2s}}+ \frac{2\abs{(u(x)-u(y))}\abs{h_{i}^{\frac{1}{2}}(x)} \abs{u(y)} \abs{h_{i}^{\frac{1}{2}}(x)-h_{i}^{\frac{1}{2}}(y)}}{\abs{x-y}^{n+2s}} \\ & \hspace{0.5cm} + \frac{\abs{u(y)}^2\abs{h_{i}^{\frac{1}{2}}(x)-h_{i}^{\frac{1}{2}}(y)}^2}{\abs{x-y}^{n+2s}}\Biggr) dxdy\Biggr) \\
		& \leq \frac{2^{\frac{2s}{n}}}{S} \Biggl(\frac{c_{n,s}}{2} \int_{T(\Omega)} \frac{\abs{(u(x)-u(y))}^2}{\abs{x-y}^{n+2s}} dxdy +  \frac{c_{n,s}}{2} \sum_{i \in I}^{} \int_{T(\Omega)}  \frac{2\abs{(u(x)-u(y))}\abs{h_{i}^{\frac{1}{2}}(x)} \abs{u(y)} \abs{h_{i}^{\frac{1}{2}}(x)-h_{i}^{\frac{1}{2}}(y)}}{\abs{x-y}^{n+2s}}dxdy \\ & \hspace{0.5cm} + \frac{c_{n,s}}{2} \sum_{i \in I}^{} \int_{T(\Omega)} \frac{\abs{u(y)}^2\abs{h_{i}^{\frac{1}{2}}(x)-h_{i}^{\frac{1}{2}}(y)}^2}{\abs{x-y}^{n+2s}}dxdy\Biggr).
	\end{align*}}
	Now, we use the following inequality to simplify the second integral on the right hand side of the above inequality. For any $a,b$ and $\lambda$ three positive real numbers, we have
	$$2ab \leq \eta a^2 + \frac{1}{\eta}b^2.$$ 
	Taking $a= \abs{(u(x)-u(y))}, b= \abs{h_{i}^{\frac{1}{2}}(x)} \abs{u(y)} \abs{h_{i}^{\frac{1}{2}}(x)-h_{i}^{\frac{1}{2}}(y)}$ and using Lemma 5.3 \cite{E1}, we get
	
	\begin{align*}
		\sum_{i \in I}^{} \norm{u^2 h_{i}}_{L^{\frac{p+1}{2}}(\Omega_{i})} & \leq \frac{2^{\frac{2s}{n}}}{S} \Biggl(\frac{c_{n,s}}{2} \int_{T(\Omega)} \frac{\abs{(u(x)-u(y))}^2}{\abs{x-y}^{n+2s}}dxdy +  \frac{c_{n,s}}{2} \sum_{i \in I}^{} \int_{T(\Omega)} \frac{\eta \abs{u(x)-u(y)}^2}{\abs{x-y}^{n+2s}} dxdy \\ & \hspace{0.5cm} + \frac{c_{n,s}}{2} \sum_{i \in I}^{} \int_{T(\Omega)} \frac{\eta^{-1}\abs{h_{i}(x)} \abs{u(y)}^2 \abs{h_{i}^{\frac{1}{2}}(x)-h_{i}^{\frac{1}{2}}(y)}^2}{\abs{x-y}^{n+2s}} dxdy  +\frac{c_{n,s}}{2} \sum_{i \in I}^{} \int_{T(\Omega)} \frac{\abs{u(y)}^2\abs{h_{i}^{\frac{1}{2}}(x)-h_{i}^{\frac{1}{2}}(y)}^2}{\abs{x-y}^{n+2s}}dxdy\ \Biggr)  \\
		& \leq \frac{2^{\frac{2s}{n}}}{S} \Biggl(\frac{c_{n,s}}{2} \int_{T(\Omega)} \frac{\abs{(u(x)-u(y))}^2}{\abs{x-y}^{n+2s}}dxdy + k\eta \frac{c_{n,s}}{2}\int_{T(\Omega)} \frac{\abs{u(x)-u(y)}^2}{\abs{x-y}^{n+2s}} dxdy \\ & \hspace{0.5cm} +\frac{ k}{\eta}C_{1}(n,s,\Omega)\int_{\Omega}u^2(y)dy +  kC_{2}(n,s,\Omega) \int_{\Omega}u^2(y)dy \Biggr),
	\end{align*}
where $C_{1}(n,s,\Omega)$ and $ C_{2}(n,s,\Omega)$ are some positive constants.
	Put $\epsilon_{0}=k\eta$ and $A(\epsilon_{0})= \frac{ k}{\eta}C_{1}(n,s,\Omega) + kC_{2}(n,s,\Omega),$ we get
	\begin{align*}
		\sum_{i \in I}^{} \norm{u^2 h_{i}}_{L^{\frac{p+1}{2}}(\Omega_{i})} \leq \frac{2^{\frac{2s}{n}}}{S}( 1+ \epsilon_{0}) \Biggl(\frac{c_{n,s}}{2} \int_{T(\Omega)} \frac{\abs{(u(x)-u(y))}^2}{\abs{x-y}^{n+2s}}dxdy + A(\epsilon_{0}) \int_{\Omega}u^2(y)dy \Biggr).
	\end{align*}
Now,
	\begin{align*}
		\norm{u}_{L^{p+1}(\Omega)}^2= \norm{u^2}_{L^{\frac{p+1}{2}}(\Omega)} &= \norm{\sum_{i \in I}^{}u^2h_{i}}_{L^{\frac{p+1}{2}}(\Omega)} \\ & \leq \sum_{i \in I}^{} \norm{u^2h_{i}}_{L^{\frac{p+1}{2}}(\Omega)}\\
		& \leq \frac{2^{\frac{2s}{n}}}{S}( 1+ \epsilon_{0}) \Biggl(\frac{c_{n,s}}{2} \int_{T(\Omega)} \frac{\abs{(u(x)-u(y))}^2}{\abs{x-y}^{n+2s}}dxdy + A(\epsilon_{0}) \int_{\Omega}u^2(y)dy \Biggr).
	\end{align*}
	With the choice of $\eta$ small enough, where $\eta= \frac{\epsilon_{0}}{k}$ so that
	$$\frac{2^{\frac{2s}{n}}}{S}( 1+ \epsilon_{0}) \leq \left(\frac{2^{\frac{2s}{n}}}{S} + \epsilon \right) $$ and 
	$$A(\epsilon)= \frac{2^{\frac{2s}{n}}}{S}( 1+ \epsilon_{0})A(\epsilon_{0}).$$ The proof of Theorem \ref{best4} is completed. \qed

\section{ Existence of minimal energy solutions }
Our objective is to find the non-constant solutions of (\ref{eqn:P1}). We begin with the preparatory definitions. Let $p=\frac{n+2s}{n-2s},$ $s\in (0,1)$ and $\Omega \subset \mathbb{R}^n$ is a bounded domain of class $C^{1}.$
\noindent Let $u\in H_{\Omega}^{s},$ define
\begin{align}
	\norm{u}_{s,\Omega,\lambda}^{2}:=& \frac{c_{n,s}}{2} \int_{T(\Omega)}\frac{{\lvert u(x)-u(y)\rvert}^{2}}{{\lvert x-y \rvert}^{n+2s}}dxdy+ \lambda\int_{\Omega} u^{2}dx.\\
	\label{energyfunctional}J_{\lambda}(u):=&\frac{1}{2}\norm{u}_{s,\Omega,\lambda}^{2}-\frac{1}{p+1}\int_{\Omega}{\lvert u \rvert}^{p+1}dx, \hspace{0.2cm} \text{an energy functional.}\\
	\label{releigh}K_{\lambda}(u):=&\frac{\norm{u}_{s,\Omega,\lambda}^{2}}{\biggl(\int_{\Omega} {\lvert u \rvert}^{p+1}dx\biggr)^{\frac{2}{p+1}}}.\\
	\label{X} X_{\lambda}(\Omega):=& \inf\bigl \{ K_{\lambda}(u): u\in H_{\Omega}^{s}\setminus \{0\} \bigr \}.
\end{align}
\begin{lem}\label{4}  Under the above notations, we have that
\begin{enumerate}
\item $X_{\lambda}>0.$ \\
\item Assume $X_{\lambda}<\frac{S}{2^{\frac{2s}{n}}},$ then there exists a 
$v \in H_{\Omega}^{s}$ such that $X_{\lambda}=K_{\lambda}(v).$ Further, define $v_{0}= X_{\lambda}^{\frac{n-2s}{4s} }v,$
then $v_{0}$ satisfies $(\ref{eqn:P2})$ and $J_{\lambda}(v_{0})< \frac{sS^{\frac{n}{2s}}}{2n}.$
\end{enumerate}
\end{lem}
\begin{proof}(1) By the fractional Sobolev embedding, there exists a constant $c>0$ such that for all $u\in H_{\Omega}^{s},$
we have
\begin{equation}
\biggl (\int_{\Omega} {\lvert u \rvert}^{p+1}dx \biggr )^{\frac{2}{p+1}} \leq c \norm {u}_{s;\Omega, \lambda}^{2}.
\end{equation}
Therefore  \begin{align*}
	X_{\lambda}= \inf \Bigl\{ K_{\lambda}(u) \mid u\in H_{\Omega}^{s}\setminus \{0\} \Bigr\} >0.
\end{align*}
(2) Let $\{ v_{k} \} $ be a minimizing sequence for  $X_{\lambda}$ such that 
$\Bigl(\int_{\Omega} {\lvert v_{k} \rvert}^{p+1} dx\Bigr)^{\frac{2}{p+1}}=1.$ Therefore, $$\norm{v_{k}}_{s,\Omega,\lambda}^2=X_{\lambda} + o(1) \text{ as } k \to \infty.$$ Thus $\left\{v_{k} \right\}$ is a bounded sequence in $H^{s}_{\Omega}.$ Therefore we can extract a subsequence of $\{v_{k} \},$ which we continue to denote by 
 $\{v_{k} \}$ itself such that $v_{k} \rightharpoonup v,$ weakly in $H_{\Omega}^{s}$ and almost everywhere in $
\Omega.$\\
We claim that $v \not \equiv 0.$ Otherwise, by Theorem \ref{best4} and compact embedding of fractional Sobolev spaces, we have
\begin{align}
\lim_{k \rightarrow \infty} \frac{c_{n,s}}{2} \int_{T(\Omega)}\frac{{\lvert v_{k}(x)-v_{k}(y)\rvert}^{2}}{{\lvert x-y \rvert}^{n+2s}}dxdy &=\lim_{k \rightarrow \infty} \norm{v_{k}}_{s,\Omega,\lambda}^{2} \nonumber \\
&= X_{\lambda} \nonumber \\
&= X_{\lambda} \cdot \left(\int_{\Omega} {\lvert v_{k} \rvert}^{p+1} dx \right)^{\frac{2}{p+1}}  \nonumber \\
& \leq  X_{\lambda} \left(\frac{2^{\frac{2s}{n}}}{S}+ \epsilon \right) \lim_{k \rightarrow \infty} \frac{c_{n,s}}{2} 
 \int_{T(\Omega)}\frac{{\lvert v_{k}(x)-v_{k}(y)\rvert}^{2}}{{\lvert x-y \rvert}^{n+2s}}dxdy. \nonumber 
\end{align}
This contradicts to our assumption $X_{\lambda}< \frac{S}{2^{\frac{2s}{n}}}.$ Hence $v \not \equiv 0.$\\ 
 Now, we claim that $K_{\lambda}(v)=X_{\lambda}.$
Let $w_{k}=v_{k}-v,$ then $w_{k}$ converges to $0$ weakly and almost everywhere in $\Omega$. 
Therefore from Proposition \ref{compact}, we have 
\begin{align}
\norm {v_{k}}_{s,\Omega,\lambda}^{2}= & \norm {v}_{s,\Omega,\lambda}^{2}+ \norm {w_{k}}_{s,\Omega,\lambda}^{2} + o(1) \nonumber \\
= & \norm {v}_{s,\Omega,\lambda}^{2} + \frac{c_{n,s}}{2}\int_{T(\Omega)}
\frac{{\lvert w_{k}(x)-w_{k}(y)\rvert}^{2}}{{\lvert x-y \rvert}^{n+2s}}dxdy + o(1), \nonumber
\end{align}
which yields
\begin{align} \label{3.14}
X_{\lambda}= & \norm {v}_{s,\Omega,\lambda}^{2} + \frac{c_{n,s}}{2}\int_{T(\Omega)}\frac{{\lvert w_{k}(x)-w_{k}(y)
\rvert}^{2}}{{\lvert x-y \rvert}^{n+2s}}dxdy + o(1).
\end{align}
Now, by Brezis-Leib Lemma \ref{BL} and Theorem \ref{best4}, we get
\begin{align}
1 = & ~{\norm{v_{k}}}_{L^{p+1}(\Omega)}^{2} \nonumber \\
=  & ~{\norm{v}}_{L^{p+1}(\Omega)}^{2} + {\norm {w_{k}}}_{L^{p+1}(\Omega)}^{2} +o(1) \nonumber \\
1 \leq & ~ {\norm{v}}_{L^{p+1}(\Omega)}^{2} +  \left(\frac{2^{\frac{2s}{n}}}{S}+ \epsilon \right) \frac{c_{n,s}}{2}
\int_{T(\Omega)}\frac{{\lvert w_{k}(x)-w_{k}(y)\rvert}^{2}}{{\lvert x-y \rvert}^{n+2s}}dxdy + o(1). \nonumber
\end{align}
 Hence
\begin{align}\label{3.15}
X_{\lambda} \leq & ~ X_{\lambda}{\norm{v}}_{L^{p+1}(\Omega)}^{2} + X_{\lambda} \left(\frac{2^{\frac{2s}{n}}}{S}+ \epsilon \right) \frac{c_{n,s}}{2} \int_{T(\Omega)}\frac{{\lvert w_{k}(x)-w_{k}(y)\rvert}^{2}}{{\lvert x-y \rvert}^{n+2s}}dxdy + o(1).
\end{align}
From (\ref{3.14}) and (\ref{3.15}), we obtain 
$$ \frac{\norm {v}_{s,\Omega,\lambda}^{2}}{{\norm{v}}_{L^{p+1}(\Omega)}^{2}}  \leq X_{\lambda}.$$
Hence $v$ minimizes $X_{\lambda}.$ Therefore $K_{\lambda}(v)=X_{\lambda}.$ 
Since $K_{\lambda}(v)=K_{\lambda}(\frac{v}{\norm{v}_{L^{p+1}(\Omega)}}),$ we assume that $\norm{v}_{L^{p+1}(\Omega)}=1.$
Now, take $v_{0}=X_{\lambda}^{\frac{n-2s}{4s}}v.$ Then
\begin{align}
J_{\lambda}(v_{0}) =  \,\,X_{\lambda}^{\frac{n}{2s}} \frac{s}{n} < \frac{sS^{\frac{n}{2s}}}{2n}\,\,
\text{(since by assumption $X_{\lambda}<\frac{S}{2^{\frac{2s}{n}}}$)}.
\end{align}
Next, we show that $v_{0}$ is a solution of (\ref{eqn:P2}). Since $K_{\lambda}(v)=X_{\lambda},$ we see that 
\begin{align} \label{uptoconstant1}
 \frac{c_{n,s}}{2}\int_{T(\Omega)}\frac{(v(x)-v(y))(w(x)-w(y))}{\abs{x-y}^{n+2s}}dxdy + 
 \lambda \int_{\Omega}vw= X_{\lambda}\int_{\Omega} \abs{v}^{p-1}vw \text{ holds, } \forall w \in H_{\Omega}^{s}.
 \end{align}
This implies that
\begin{align}\label{uptoconstant2}
  \frac{c_{n,s}}{2}\int_{T(\Omega)}\frac{(v_{0}(x)-v_{0}(y))(w(x)-w(y))}{\abs{x-y}^{n+2s}}dxdy + 
  \lambda \int_{\Omega}v_{0}w=\int_{\Omega} \abs{v_{0}}^{p-1}v_{0}w \,\, \text{holds, } \forall w \in H_{\Omega}^{s}.
 \end{align}
Thus in view of Definition \ref{weak}, we see that $v_{0}$ a solution of (\ref{eqn:P2}). This completes the proof.
\end{proof}

\noindent In the next proposition, we obtain the nonlocal version of equations $(1.11),~ (1.12)$ and $(1.13)$ 
of Brezis-Nirenberg \cite{B3}. This proposition help us to prove the assumption made in previous lemma on $X_{\lambda}.$ 

\begin{prop}\label{BN}
Let $\psi \in C^{\infty}_{0}(B_{\frac{R}{2}}(0))$ be a nonnegative  radial function s.t. 
\[ \psi(x)= \begin{cases}
1 & \text{if $\abs{x}\leq \frac{R}{4}$}, \\
0 &  \text{if $\abs{x} > \frac{R}{2}$},
\end{cases} \]  
and for $\epsilon > 0,$ define 
\begin{equation}
V_{\epsilon}(x):= \frac{\psi(x)}{(\epsilon + \abs{x}^{2})^{\frac{(n-2s)}{2}}}.
\end{equation}
Then, we have
\begin{align} \label{3.19}
 \frac{c_{n,s}}{2} \int_{\Omega}\int_{\Omega}\frac{{\lvert V_{\epsilon}(x)-V_{\epsilon}(y)\rvert}^{2}}
{{\lvert x-y \rvert}^{n+2s}}dxdy = & ~\frac{L_{1}}{\epsilon^{\frac{(n-2s)}{2}}}+ O(1).
\end{align}
\begin{align}\label{3.20}
 \biggl(\int_{\Omega} \abs{V_{\epsilon}}^{p+1} \biggr)^{2/p+1} = & ~\frac{L_{2}}{\epsilon^{\frac{(n-2s)}{2}}} + O(1).
\end{align}
\begin{align}\label{3.21}
{\int_{\Omega}\abs{V_{\epsilon}}^{2}} =
\begin{cases}
\frac{L_{3}}{\epsilon^{\frac{(n-4s)}{2}}} +O(1)  \text{ when }n > 4s, \\
L_{3}\abs{log\epsilon}+O(1) \text{ when } n=4s=2,3, \\
L_{3}sinh^{-1}(\frac{r}{\sqrt{\epsilon}})+O(1) \text{ when } n=4s=1, \\
\end{cases}
\end{align}
where $r>0$ is some constant and $L_{1},~L_{2}~and~ L_{3}$ are positive constants such 
that $\frac{L_{1}}{L_{2}}=S,$ where $S$ is defined in Theorem \ref{Best}. 
\end{prop}
\begin{proof} Since $\psi$ is $1$ in a neighbourhood of $0,$ we have the following
\begin{align*}
\frac{c_{n,s}}{2} \int_{\Omega}\int_{\Omega}\frac{{\lvert V_{\epsilon}(x)-V_{\epsilon}(y)
\rvert}^{2}}{{\lvert x-y \rvert}^{n+2s}}dxdy &= \frac{c_{n,s}}{2} \int_{\Omega} \int_{\Omega} \frac{\abs{\frac{\psi(x)}{(\epsilon+
\abs{x}^{2})^{\frac{(n-2s)}{2}}}-\frac{\psi(y)}{(\epsilon+\abs{y}^{2})^{\frac{(n-2s)}{2}}}}^{2}}{\abs{x-y}^{n+2s}}dxdy \\
&=\frac{c_{n,s}}{2} \int_{\Omega} \int_{\Omega}  \frac{\abs{\left(\frac{\psi(x)-1}{(\epsilon+\abs{x}^{2})^{\frac{(n-2s)}{2}}}- \frac{\psi(y)-1}{(\epsilon+\abs{y}^{2})^{\frac{(n-2s)}{2}}} \right) + \left(\frac{1}{(\epsilon+\abs{x}^{2})^{\frac{(n-2s)}{2}}}- \frac{1}{(\epsilon+\abs{y}^{2})^{\frac{(n-2s)}{2}}} \right)}^2}{\abs{x-y}^{n+2s}} dxdy \\
& = \frac{c_{n,s}}{2}\int_{\mathbb{R}^{n}} \int_{\mathbb{R}^{n}} 
\frac{\abs{\frac{1}{(\epsilon+\abs{x}^{2})^{\frac{(n-2s)}{2}}}-\frac{1}{(\epsilon+\abs{y}^{2})^{\frac{(n-2s)}{2}}}}^{2}}
{\abs{x-y}^{n+2s}}dxdy + O(1). \nonumber
\end{align*}
The change of variables $$x'=\frac{x}{\sqrt{\epsilon}},\,\,\, y'=\frac{y}{\sqrt{\epsilon}}$$ gives us 
\begin{align*}
\frac{c_{n,s}}{2} \int_{\Omega}\int_{\Omega}\frac{{\lvert V_{\epsilon}(x)-V_{\epsilon}(y)\rvert}^{2}}{{\lvert x-y \rvert}^{n+2s}}dxdy & = \frac{1}{{\epsilon}^{\frac{(n-2s)}{2}}} \left(\frac{c_{n,s}}{2 }   \int_{\mathbb{R}^{n}} \int_{\mathbb{R}^{n}} \frac{\abs{\frac{1}{(1+\abs{x'}^{2})^{\frac{(n-2s)}{2}}}-\frac{1}{(1+\abs{y'}^{2})^{\frac{(n-2s)}{2}}}}^{2}}{\abs{x'-y'}^{n+2s}}dx'dy'\right) + O(1) \nonumber \\
& = \frac{L_{1}}{{\epsilon}^{\frac{(n-2s)}{2}}} + O(1), \nonumber
\end{align*}
where $$L_{1}= \frac{c_{n,s}}{2} \int_{\mathbb{R}^{n}} \int_{\mathbb{R}^{n}} \frac{\abs{\frac{1}{(1+\abs{x'}^{2})^{\frac{(n-2s)}{2}}}-\frac{1}{(1+\abs{y'}^{2})^{\frac{(n-2s)}{2}}}}^{2}}{\abs{x'-y'}^{n+2s}}dx'dy'.$$
This verifies $(\ref{3.19}).$ Further, we have
\begin{align*}
\int_{\Omega} \abs{V_{\epsilon}}^{p+1} = & \int_{\Omega} \frac{{\psi}^{p+1}(x)}{(\epsilon+ {\abs{x}}^{2})^{n}}dx \\ 
= & \int_{\Omega} \frac{[{\psi}^{p+1}(x)-1]}{(\epsilon+ {\abs{x}}^{2})^{n}}dx + \int_{\Omega} \frac{1}{(\epsilon+ {\abs{x}}^{2})^{n}}dx \\
= &~ O(1)+ \int_{\mathbb{R}^{n}} \frac{1}{(\epsilon+ {\abs{x}}^{2})^{n}}dx\\
= & ~\frac{L'_{2}}{\epsilon^{n/2}} + O(1),
\end{align*}
 \text{where} $$L'_{2}=\int_{\mathbb{R}^{n}} \frac{1}{(1+ {\abs{x}}^{2})^{n}}dx.$$
Therefore
 \begin{align*}
 \biggl(\int_{\Omega} \abs{V_{\epsilon}}^{p+1} \biggr)^{\frac{2}{p+1}} =&~ \frac{L_{2}}{\epsilon^{\frac{(n-2s)}{2}}} + O(1),
\end{align*}
where $$L_{2}= {\norm{\left(1 + \abs{x}^2 \right)^{\frac{-(n-2s)}{2}}}}_{L^{p+1}(\mathbb{R}^{n})}^{2}$$ and thus $(\ref{3.20})$ is verified. Since $\left(\epsilon + \abs{x}^2 \right)^{\frac{-(n-2s)}{2}}$ gives the equality in (\ref{Best3}), one can note that
 $$\frac{L_{1}}{L_{2}}=S.$$
Now, we verify (\ref{3.21}). For this,
\begin{align*}
\int_{\Omega} \abs{V_{\epsilon}}^{2} = & \int_{\Omega} \frac{[{\psi}^{2}(x)-1]}{(\epsilon+ {\abs{x}}^{2})^{n-2s}}dx + \int_{\Omega} \frac{1}{(\epsilon+ {\abs{x}}^{2})^{n-2s}}dx \\
= ~ & O(1) + \int_{\Omega} \frac{1}{(\epsilon+ {\abs{x}}^{2})^{n-2s}}dx.
\end{align*}
When $n > 4s,$ we have
\begin{align*}
\int_{\Omega} \frac{1}{(\epsilon+ {\abs{x}}^{2})^{n-2s}}dx = &~ \int_{\mathbb{R}^{n}} \frac{1}{(\epsilon+ {\abs{x}}^{2})^{n-2s}}dx + O(1)\\
= &~ \frac{1}{\epsilon^{n-2s}}\int_{\mathbb{R}^{n}} \frac{1}{(1+ \abs{\frac{x}{\sqrt{\epsilon}}}^{2})^{n-2s}}dx + O(1).
\end{align*}
The change of variable $$\frac{x}{\sqrt{\epsilon}}=y,$$ gives us
\begin{align*}
\int_{\Omega} \frac{1}{(\epsilon+ {\abs{x}}^{2})^{n-2s}}dx =&~\frac{1}{\epsilon^{\frac{(n-2s)}{2}}}\int_{\mathbb{R}^{n}} 
\frac{1}{(1+ {\abs{y}}^{2})^{n-2s}}dx + O(1).
\end{align*}

\noindent This verifies (\ref{3.21}) with  
\[ L_{3}=\int_{\mathbb{R}^{n}} \frac{1}{(1+ {\abs{y}}^{2})^{n-2s}}dx= \begin{cases}
\frac{\sqrt{\pi}\Gamma(\frac{1-4s}{2})}{\Gamma(1-2s)} & \text{if $n=1$}, \\
\frac{\omega_{n}\Gamma(\frac{n-4s}{2}\Gamma(\frac{n}{2}))}{2\Gamma(n-2s)} &  \text{if $n>1$},
\end{cases} \]  \\
where $\omega_{n}$ is the surface measure of unit sphere in $\mathbb{R}^{n}.$\\
\noindent And when $n=4s,$ we see that
\begin{equation}
\int_{\abs{x}\leq r_{1}} \frac{1}{(\epsilon+\abs{x}^2)^{2s}} \leq \int_{\Omega} \frac{1}{(\epsilon+
\abs{x}^2)^{2s}} \leq \int_{\abs{x}\leq r_{2}} \frac{1}{(\epsilon+\abs{x}^2)^{2s}},
\end{equation}
for some positive constants $r_{1}$ and $r_{2}.$
Therefore
\[ \int_{\abs{x}\leq r} \frac{1}{(\epsilon+\abs{x}^2)^{2s}} = \begin{cases}
\omega_{n}\int_{0}^{r}\frac{t^{4s-1}}{\epsilon+t^{4s}}dt+ O(1) & \text{if $n=2,3,$} \\
2sinh^{-1}(\frac{r}{\sqrt{\epsilon}})& \text{if $n=1$}, \end{cases} \]
for some positive constant $r.$ This entails
 \[\int_{\abs{x}\leq r} \frac{1}{(\epsilon+\abs{x}^2)^{2s}}=\begin{cases}
\frac{\omega_{n}}{4s}\abs{log\epsilon}+ O(1)& \text{if $n=2,3,$} \\
2sinh^{-1}(\frac{r}{\sqrt{\epsilon}})& \text{if $n=1$}.
\end{cases} \] \\

\noindent So this verifies (\ref{3.21}) with $$L_{3}=\frac{\omega_{n}}{n},$$ when $n=2,3$ and $L_{3}=2,$ when $n=1.$
\end{proof}
\noindent Now, we use the above Proposition \ref{BN} to prove following lemma.

\begin{lem} \label{P} Let $\Omega$ be a bounded domain of class $C^{1}.$ Then for every $\lambda >0,$ we have $ X_{\lambda}< \frac{S}{2^{\frac{2s}{n}}}.$
\end{lem}
\begin{proof}
Since $\Omega$ is a bounded domain in $\mathbb{R}^n,$ we may assume that $\Omega$ lies to the one side of some hyperplane in $\mathbb{R}^n.$ Without loss of generality, assume that $\Omega \subset \mathbb{R}^{n}_{+}:= \left\{(x_{1},x_{2},\dots,x_{n}) \in \mathbb{R}^n \mid x_{n}>0 \right\}.$
Therefore, we get
\begin{align} \label{3.24}
\int_{\Omega} \int_{\Omega} \frac{{\lvert V_{\epsilon}(x)-V_{\epsilon}(y)\rvert}^{2}}{{\lvert x-y \rvert}^{n+2s}}dydx & \leq  \int_{\mathbb{R}^{n}_{+}} \int_{\mathbb{R}^{n}_{+}}
\frac{{\lvert V_{\epsilon}(x)-V_{\epsilon}(y)\rvert}^{2}}{{\lvert x-y \rvert}^{n+2s}}dydx \nonumber \\ & \leq \frac{1}{4} \int_{\mathbb{R}^{n}} \int_{\mathbb{R}^{n}}
\frac{{\lvert V_{\epsilon}(x)-V_{\epsilon}(y)\rvert}^{2}}{{\lvert x-y \rvert}^{n+2s}}dydx.
\end{align}
Since $\Omega$ is an open set, for $x \in \Omega$ there exists some $r>0$ such that $B_{r}(x) \subset \Omega.$ Using $V_{\epsilon} \in C_{0}^{\infty}(\mathbb{R}^n)$ for each $\epsilon >0$ and H\"{o}lder's inequality, we get
\begin{align*} 
	\int_{\Omega} \int_{\mathcal{C}\Omega}\frac{{\lvert V_{\epsilon}(x)-V_{\epsilon}(y)\rvert}^{2}}{{\lvert x-y \rvert}^{n+2s}}dydx &  \leq  \int_{\Omega} \int_{ \mathcal{C}\Omega}\frac{2{\lvert V_{\epsilon}(x)\rvert}^{2}}{{\lvert x-y \rvert}^{n+2s}}dydx +  \int_{\Omega} \int_{\mathcal{C} \Omega}\frac{2{\lvert V_{\epsilon}(y)\rvert}^{2}}{{\lvert x-y \rvert}^{n+2s}}dydx   \nonumber \\
	& \leq \int_{\Omega} \int_{\mathcal{C}B_{r}(x) }\frac{2{\lvert V_{\epsilon}(x)\rvert}^{2}}{{\lvert x-y \rvert}^{n+2s}}dydx +  \int_{\Omega} \int_{\mathcal{C}B_{r}(x) }\frac{2{\lvert V_{\epsilon}(y)\rvert}^{2}}{{\lvert x-y \rvert}^{n+2s}}dydx \nonumber \\
	& \leq \int_{\Omega} {\lvert V_{\epsilon}(x)\rvert}^{2} \left( \int_{\mathcal{C} B_{r}(x) }\frac{2}{{\lvert x-y \rvert}^{n+2s}}dy\right) dx +  \int_{\Omega} \int_{\mathcal{C} B_{r}(x) }\frac{2{\lvert V_{\epsilon}(y)\rvert}^{2}}{{\lvert x-y \rvert}^{n+2s}}dydx \nonumber \\
	& \leq \int_{\Omega} {\lvert V_{\epsilon}(x)\rvert}^{2} \left( \int_{\mathcal{C} B_{r}(x) }\frac{2}{{\lvert x-y \rvert}^{n+2s}}dy\right) dx + \frac{1}{\epsilon^{(n-2s)}}\int_{\Omega} \int_{\mathcal{C} B_{r}(x) }\frac{2 dydx }{\left(1+\abs{\frac{y}{\sqrt{\epsilon}}}^2\right)^{(n-2s)}{\lvert x-y \rvert}^{n+2s}}\nonumber \\
	& \leq \int_{\Omega} {\lvert V_{\epsilon}(x)\rvert}^{2} \left( \int_{r}^{\infty}\frac{ 2 \omega_{n} \rho^{n-1}}{{\rho}^{n+2s}}d\rho \right) dx \nonumber \\ & \hspace{1cm} + \frac{1}{\epsilon^{(n-2s)}} \bigints_{\Omega} \left(\int_{\mathcal{C} B_{r}(x)} \frac{1}{\left(1+\abs{\frac{y}{\sqrt{\epsilon}}}^2\right)^{(2n-4s)}}dy \right)^{\frac{1}{2}} \left(\int_{ \mathcal{C}B_{r}(x) }\frac{4}{{\lvert x-y \rvert}^{2n+4s}}dy\right)^{\frac{1}{2}} dx .
	\end{align*}
By the change of variable $$\frac{y}{\sqrt{\epsilon}} = z,$$ we get for $n > \frac{8s+2}{3}$
\begin{align}\label{3.25}
	\int_{\Omega} \int_{\mathcal{C}\Omega}\frac{{\lvert V_{\epsilon}(x)-V_{\epsilon}(y)\rvert}^{2}}{{\lvert x-y \rvert}^{n+2s}}dydx &  \leq \frac{\omega_{n}}{sr^{2s}}\int_{\Omega} {\lvert V_{\epsilon}(x)\rvert}^{2} dx + \epsilon^{2s} \int_{\Omega} \left(\int_{\mathbb{R}^n} \frac{dz}{(1+\abs{z}^2)^{(2n-4s)}}\right)^{\frac{1}{2}} \left( \int_{r}^{\infty}\frac{ 4 \omega_{n} \rho^{n-1}}{{\rho}^{2n+4s}}d\rho \right)^{\frac{1}{2}} dx \nonumber \\
	&  \leq \frac{\omega_{n}}{sr^{2s}}\int_{\Omega} {\lvert V_{\epsilon}(x)\rvert}^{2} dx + \epsilon^{2s} \int_{\Omega} \left(\int_{0}^{\infty} \frac{\omega_{n} \rho^{n-1}}{(1+ \rho^2)^{(2n-4s)}}\right)^{\frac{1}{2}} \left( \int_{r}^{\infty}\frac{ 4 \omega_{n} \rho^{n-1}}{{\rho}^{2n+4s}}d\rho \right)^{\frac{1}{2}} dx \nonumber \\
	&  \leq    \frac{\omega_{n}}{sr^{2s}}\int_{\Omega} {\lvert V_{\epsilon}(x)\rvert}^{2} dx + \epsilon^{2s} \int_{\Omega} \left( \frac{\omega_{n} \Gamma{(\frac{n+2}{2}) \Gamma(\frac{3n-8s-2}{2})} }{2\Gamma(2n-4s)} \right)^{\frac{1}{2}} \left(\frac{4 \omega_{n}}{(n+4s)r^{n+2s}} \right)^{\frac{1}{2}} dx  \nonumber \\
	&  \leq    \frac{\omega_{n}}{sr^{2s}}\int_{\Omega} {\lvert V_{\epsilon}(x)\rvert}^{2} dx + \epsilon^{2s} \abs{\Omega} \left( \frac{\omega_{n} \Gamma{(\frac{n+2}{2}) \Gamma(\frac{3n-8s-2}{2})} }{2\Gamma(2n-4s)} \right)^{\frac{1}{2}} \left(\frac{4 \omega_{n}}{(n+4s)r^{n+2s}} \right)^{\frac{1}{2}}. 
\end{align}
\noindent Combining equations (\ref{3.24}) and (\ref{3.25}), we get
\begin{align}\label{3.26}
\frac{c_{n,s}}{2}\int_{T(\Omega)}\frac{{\lvert V_{\epsilon}(x)-V_{\epsilon}(y)\rvert}^{2}}{{\lvert x-y \rvert}^{n+2s}}dydx &= \frac{c_{n,s}}{2}\int_{\Omega} \int_{\Omega} \frac{{\lvert V_{\epsilon}(x)-V_{\epsilon}(y)\rvert}^{2}}{{\lvert x-y \rvert}^{n+2s}}dydx + c_{n,s} \int_{\Omega} \int_{\mathcal{C}\Omega}\frac{{\lvert V_{\epsilon}(x)-V_{\epsilon}(y)\rvert}^{2}}{{\lvert x-y \rvert}^{n+2s}}dydx \nonumber \\
& \leq  \frac{1}{4} \left(\frac{c_{n,s}}{2} \int_{\mathbb{R}^{n}} \int_{\mathbb{R}^{n}}
\frac{{\lvert V_{\epsilon}(x)-V_{\epsilon}(y)\rvert}^{2}}{{\lvert x-y \rvert}^{n+2s}}dydx \right)  + \frac{c_{n,s}\omega_{n}}{sr^{2s}} \int_{\Omega} {\lvert V_{\epsilon}(x)\rvert}^{2} dx  \nonumber \\
& \hspace{1cm} + \epsilon^{2s} \abs{\Omega} \left( \frac{4\omega_{n}^2 \abs{\Omega}^2 \Gamma{(\frac{n+2}{2}) \Gamma(\frac{3n-8s-2}{2})} }{2 (n+4s)r^{n+2s}\Gamma(2n-4s)} \right)^{\frac{1}{2}}.
\end{align}
Now, using (\ref{3.26}), (\ref{3.19}) and (\ref{3.21}), we have for $\epsilon \rightarrow 0,$
\begin{align} \label{3.022}
\frac{c_{n,s}}{2}\int_{T(\Omega)}\frac{{\lvert V_{\epsilon}(x)-V_{\epsilon}(y)\rvert}^{2}}{{\lvert x-y \rvert}^{n+2s}}dydx=& \frac{L_{1}}{4\epsilon^{\frac{(n-2s)}{2}}}+ \frac{C(n,s,r)L_{3}}{\epsilon^{\frac{(n-4s)}{2}}} + O(1)
\end{align}
Thus from Equations (\ref{3.022}), (\ref{3.20}) and (\ref{3.21}) for $n> \max \left\{4s, \frac{8s+3}{2} \right\},$ we obtain 
$$K_{\lambda}(V_{\epsilon})  < \frac{S}{2^{\frac{2s}{n}}},$$
provided $\epsilon >0$ is small enough.
Hence 
\begin{align*}
X_{\lambda} =\inf\bigl \{ K_{\lambda}(u) \mid u\in H_{\Omega}^{s}\setminus \{0\} \bigr \} < \frac{S}{2^{\frac{2s}{n}}}. 
\end{align*} 
\end{proof}

\noindent Combining the above lemmas, we are ready to prove the main result of this section. \\

\noindent \textbf{Proof of Theorem \ref{2}:} The Lemma \ref{4} and Lemma \ref{P}, 
gives us a solution $v_{0}$ of the problem (\ref{eqn:P2}) with $J_{\lambda}(v_{o})<\frac{s(S)^{\frac{n}{2s}}}{2n}.$
Now, we show that $v_{0}$ is a nonconstant solution. For this, 
define $$\lambda^{*}=\frac{S}{(2\lvert \Omega \rvert)^{\frac{2s}{n}}},$$
where by $\lvert \Omega \rvert,$ we mean measure of $\Omega.$
Then for $\lambda > \lambda^{*}$ and $v_{*}=\lambda^{\frac{(n-2s)}{4s}},$ we have
\begin{align}\label{j1}
J_{\lambda}(v_{1}) &= \frac{1}{2}\int_{\Omega} \lambda^{\frac{(n-2s+2s)}{2s}}-\frac{(n-2s)}{2n}\int_{\Omega}\lambda^{\frac{n}{2s}} \nonumber\\
&= \frac{1}{2}\lvert \Omega \rvert \lambda^{\frac{n}{2s}}-\frac{n-2s}{2n}\lambda^{\frac{n}{2s}}\lvert \Omega \rvert \nonumber\\
&= \frac{s}{n}\lambda^{\frac{n}{2s}} \lvert \Omega \rvert \nonumber \\
&> \frac{s(S)^{\frac{n}{2s}}}{2n}.
\end{align}
Thus $v_{0}$ is a nonconstant function. Otherwise, if $v_0$ is constant, then it implies that 
$v_{0}=\lambda^{\frac{n-2s}{4s}}$, which contradicts to (\ref{j1}). \\
Further, we have
$$\abs{\abs{u(x)}-\abs{u(y)}} \leq \abs{u(x)-u(y)} $$
and hence
$$ \frac{c_{n,s}}{2}\int_{T(\Omega)}\frac{{\lvert \abs{u(x)}-\abs{u(y)}\rvert}^{2}}{{\lvert x-y \rvert}^{n+2s}}dxdy \leq ~ \frac{c_{n,s}}{2} \int_{T(\Omega)}\frac{{\lvert u(x)-u(y)\rvert}^{2}}{{\lvert x-y \rvert}^{n+2s}}dxdy,$$ which yields
$$ K_{\lambda}(\abs{u}) \leq K_{\lambda}(u),\,\,\forall~u\in H_{\Omega}^{s}.$$
It means that if $u$ is a minimizer, then $\abs{u}$ is also a minimizer for $X_{\lambda}.$
Therefore, we can assume that $v_{0} \geq 0$ in $\Omega.$ Note that $v_{0}$ is not identically zero as already proved in Lemma \ref{4}.  \qed

\section*{Appendix: Uniqueness of minimal energy solutions in critical case}
We show the uniqueness of minimal energy solutions for small domains using the same approach as in \cite{J2}.
Let $\Omega \subset \mathbb{R}^{n}$ be a bounded domain of class $C^{1}$. Consider the problem 
\begin{equation}\label{eqn:P3}
 \left\{\begin{array}{l l} { (-\Delta)^{s}u+ \lambda u= X_{\lambda}(\Omega)\abs{u}^{p-1}u } & \text{in $\Omega$ ,  } \\ 
\hspace{0.9cm} { \mathcal{N}_{s}u(x)=0 } & \text{in $\mathbb{R}^{n}\setminus \overline{\Omega}$,} \\
\hspace{1.8cm} {u \geq 0} &\text{in $\Omega,$}\end{array} \right.
\end{equation}
where $\lambda >0$ is a constant, $p=\frac{n+2s}{n-2s},~ n>\max \left\{ 4s, \frac{8s+2}{3}\right\},~ s\in (0,1)$ and $X_{\lambda}(\Omega)$ 
is defined earlier in (\ref{X}). \\ 
\indent In the proof of Lemma \ref{4} (see, Equations (\ref{uptoconstant1}) and (\ref{uptoconstant2})),
we have seen that minimal energy solutions of (\ref{eqn:P2}) and (\ref{eqn:P3}) are same up to a constant. 
Therefore, in order to check the uniqueness of minimal energy solutions  to (\ref{eqn:P2}), it is enough to check it 
for the above problem (\ref{eqn:P3}).

We consider a family of contracted domains \begin{align}
 \Omega_{\eta}:=\eta\cdot \Omega=\bigl \{\eta x: x\in \Omega \bigr \}
 \end{align}
and look for the asymptotic behaviour of minimal energy  solutions to the above problem in 
$\Omega_{\eta}$ as $\eta \rightarrow 0$.
The following results have the similar arguments as in Fern\'{a}ndez Bonder et al.  \cite{J2} with some modifications to handle the critical nonlinearity.
Fern\'{a}ndez Bonder et al.  \cite{J2} established the uniqueness of the minimal energy  solutions in subcritical case. 
Here, one cannot directly apply compact Sobolev embedding as was used in \cite{J2}.
In our setup, we use  Brezis-Lieb lemma, Theorem \ref{Best3} and Theorem \ref{best4} in the proof of Lemma \ref{minimizer} 
to overcome the difficulties 
occurring  due to the lack of compactness in embedding.
We begin with the following lemma.
\begin{lem}\label{minimizer4.1} (Lemma 3.1,\,\cite{J2})
Under the above notations, we have 
\begin{align*}
X_{\lambda}(\Omega_{\eta}) \leq \lambda \abs{\Omega_{\eta}}^{\frac{p-1}{p+1}}=\lambda \eta^{n \bigl(\frac{p-1}{p+1}\bigr)}\abs{\Omega}^{\frac{p-1}{p+1}}.
\end{align*}
\end{lem}

The following observations are useful to prove the next lemma. First,  we define for any $v\in H_{\Omega_{\eta}}^{s},$
$$ ({v})_{s, \Omega}^{2}:=\frac{c_{n,s}}{2}\int_{T(\Omega)}\frac{\abs{v(x)-v(y)}^{2}}{\abs{x-y}^{n+2s}}dxdy. $$
Let $v\in H_{\Omega_{\eta}}^{s},$ if we denote $\bar{v}(x)=v(\eta x),$ then $\bar{v} \in H_{\Omega_{\eta}}^{s}.$
Moreover, 
\begin{align}
(\bar{v})_{s, \Omega}^{2}&=\frac{c_{n,s}}{2}\int_{T(\Omega)}\frac{\abs{v(\eta x)-v(\eta y)}^{2}}{\abs{x-y}^{n+2s}}dxdy \nonumber \\ &= \frac{c_{n,s}}{2} \biggl[\int_{\Omega}\int_{\Omega} \frac{\abs{v(\eta x)-v(\eta y)}^{2}}{\abs{x-y}^{n+2s}}dxdy + 2\int_{\Omega}\int_{\mathbb{R}^{n} \setminus \Omega} \frac{\abs{v(\eta x)-v(\eta y)}^{2}}{\abs{x-y}^{n+2s}}dxdy \biggr]. \nonumber \\
\text{Changing variables $\eta x=z $ and $\eta y=w$ gives us} \nonumber \\
(\bar{v})_{s, \Omega}^{2} &=\eta^{2s-n} \frac{c_{n,s}}{2} \biggl[\int_{\Omega_{\eta}}\int_{\Omega_{\eta}} 
\frac{\abs{v(z)-v(w)}^{2}}{\abs{z-w}^{n+2s}}dzdw + 2\int_{\Omega_{\eta}}\int_{\mathbb{R}^{n}\setminus \Omega_{\eta}} 
\frac{\abs{v(z)-v(w)}^{2}}{\abs{z-w}^{n+2s}}dzdw \biggr]. \nonumber \\
&= \eta^{2s-n} ({v})_{s, \Omega_{\eta}}^{2}.
\end{align}
And
\begin{align}
\norm{\bar{v}}_{L^{p+1}(\Omega)}=&~\eta^{-\frac{n}{p+1}}\norm{{v}}_{L^{p+1}(\Omega_{\eta})}.
\end{align}
Therefore 
\begin{align}
\frac{\norm{v}_{s,\Omega_{\eta},\lambda}^{2}}{\norm{{v}}_{L^{p+1}(\Omega_{\eta})}^{2}}= &~ 
\eta^{n \bigl(\frac{p-1}{p+1}\bigr)} \frac{\eta^{-2s}(\bar{v})_{s, \Omega}^{2}+
\lambda \norm{\bar{v}}_{L^{2}(\Omega)}^{2} }{\norm{\bar{v}}_{L^{p+1}(\Omega)}^{2}},
\end{align}
where $${\norm{v}}^{2}_{s,\Omega_{\eta},\lambda}= \frac{c_{n,s}}{2} \int_{T(\Omega_{\eta})}
\frac{{\lvert v(x)-v(y)\rvert}^{2}}{{\lvert x-y \rvert}^{n+2s}}dxdy+ \lambda\int_{\Omega_{\eta}} v^{2}dx.$$
\begin{lem}\label{minimizer}
Let $v_{\eta}\in H_{\Omega_{\eta}}^{s}$ be a minimizer for $X_{s}(\Omega_{\eta}).$
Then the rescaled minimizers $\bar{v}_{\eta}(x):= v_{\eta}(\eta x)$ normalized such that 
$\norm{\bar{v}_{\eta}}_{L^{p+1}(\Omega)}=1$ yields
$$ \bar{v}_{\eta} \rightarrow \abs{\Omega}^{-\frac{1}{p+1}} \text{ strongly in $H_{\Omega}^{s}$ }.$$ Moreover, we have 
$$ \lim_{\eta \rightarrow 0} \frac{X_{\lambda}(\Omega_{\eta})}{\eta^{n \bigl(\frac{p-1}{p+1}\bigr)}}=\abs{\Omega}^{\frac{p-1}{p+1}}.$$
\end{lem}
\begin{proof} From the above observations, we have the following 
$$ X_{\lambda}(\Omega_{\eta}) = \eta^{n \bigl(\frac{p-1}{p+1}\bigr)} \frac{\eta^{-2s}(\bar{v}_{\eta})_{s, \Omega}^{2}+\lambda \norm{\bar{v}_{\eta}}_{L^{2}(\Omega)}^{2} }{\norm{\bar{v}_{\eta}}_{L^{p+1}(\Omega)}^{2}}.$$ 
Now, by Lemma \ref{minimizer4.1}, we get 
\begin{align}\label{4.6}
\frac{\eta^{-2s}(\bar{v}_{\eta})_{s, \Omega}^{2}+\lambda \norm{\bar{v}_{\eta}}_{L^{2}(\Omega)}^{2} }{\norm{\bar{v}_{\eta}}_{L^{p+1}(\Omega)}^{2}} \leq &~ \lambda \abs{\Omega}^{\frac{p-1}{p+1}}.
\end{align}
Let us now fix that $\norm{\bar{v}_{\eta}}_{L^{p+1}(\Omega)}^{2}=1.$
From the above inequality, it follows that $\bar{v}_{\eta}$ is uniformly bounded in $H_{\Omega}^{s}.$
Therefore, upto some subsequence $\eta_{k} \rightarrow 0,$ we have  
\begin{align} 
\label{4.7}&v_{\eta} \rightharpoonup \bar{v}~~~ \text{weakly in } H_{\Omega}^{s},\\
\label{4.8}&v_{\eta} \rightarrow \bar{v}~~~ \text{strongly in } L^{q}(\Omega),~\text{for}~1\leq q < p+1.
\end{align}
We know that the embedding $H^{s}({\Omega})\hookrightarrow  L^{p+1}(\Omega)$ is not compact. 
Therefore we cannot deduce directly that the above sequence converges strongly in $L^{p+1}(\Omega).$ 
Since norm is weakly lower semi-continuous, from (\ref{4.6}) and (\ref{4.7}), we have that 
$$ (\bar{v})_{s, \Omega}^{2} \leq \liminf_{\eta \rightarrow 0}(\bar{v_{\eta}})_{s, \Omega}^{2}=0.$$ 
Therefore $\bar{v}$ is a constant. Also, we have that 
\begin{align}\label{4.9}
  \norm{\bar{v}}_{L^{p+1}(\Omega)}^{2} \leq~ \liminf_{\eta \rightarrow 0}\norm{\bar{v}_{\eta}}_{L^{p+1}(\Omega)}^{2}=1.
\end{align}
Now, let $\bar{w}_{\eta}=\bar{v}_{\eta}- \bar{v},$ then $\bar{w}_{\eta} \rightharpoonup 0$ weakly 
in $L^{p+1}(\Omega)$ and almost everywhere in $\Omega.$ Therefore, again using Brezis-Lieb lemma \ref{BL}, Proposition \ref{compact} and Theorem \ref{best4}, we get for $\eta \rightarrow 0,$
\begin{align}\label{4.10}
1 &= \norm{\bar{v}_{\eta}}_{L^{p+1}(\Omega)}^{2} \nonumber \\
& = \norm{\bar{v}}_{L^{p+1}(\Omega)}^{2} + \norm{\bar{w}_{\eta}}_{L^{p+1}(\Omega)}^{2} + o(1) \nonumber \\
& \leq  \norm{\bar{v}}_{L^{p+1}(\Omega)}^{2} + \left(\frac{2^{\frac{2s}{n}}}{S} + \epsilon\right) \frac{c_{n,s}}{2}
\int_{T(\Omega)}\frac{\abs{\bar{w}_{\eta}(x)-\bar{w}_{\eta}(y)}^{2}}{\abs{x-y}^{n+2s}}+ o(1) \nonumber \\
& =  \norm{\bar{v}}_{L^{p+1}(\Omega)}^{2} + \left(\frac{2^{\frac{2s}{n}}}{S} + \epsilon\right) \frac{c_{n,s}}{2} \int_{T(\Omega)}
\frac{\abs{\bar{v}_{\eta}(x)-\bar{v}_{\eta}(y)}^{2}}{\abs{x-y}^{n+2s}}+ o(1)\text{ (since $\bar{v}$ is a constant).} \nonumber \\
& = \norm{\bar{v}}_{L^{p+1}(\Omega)}^{2} + o(1) ~ (\text{since by (\ref{4.6})}). 
\end{align}
Therefore from (\ref{4.9}) and (\ref{4.10}), we see that $\norm{\bar{v}}_{L^{p+1}(\Omega)}^{2}=1.$ 
This implies that $$\bar{v}= \abs{\Omega}^{\frac{-1}{p+1}}.$$ From these estimates, we can easily conclude that 
\begin{align*}
\lambda \abs{\Omega}^{\frac{p-1}{p+1}}\leq & \liminf_{\eta \rightarrow 0}\eta^{-2s}(\bar{v_{\eta}})_{s, \Omega}^{2}+\lambda \norm{\bar{v}{_{\eta}}}_{L^{2}(\Omega)}^{2} \\
=& \liminf_{\eta \rightarrow 0} \frac{X_{\lambda}(\Omega_{\eta})}{\eta^{n\big(\frac{p-1}{p+1}\big)}} \leq \limsup_{\eta \rightarrow 0} \frac{X_{\lambda}(\Omega_{\eta})}{\eta^{n\big(\frac{p-1}{p+1}\big)}} \leq \lambda \abs{\Omega}^{\frac{p-1}{p+1}}.
\end{align*}
This completes the proof.
\end{proof}

The next theorem gives us the uniqueness of minimizers if the domain is contracted enough. 
Let us define the space
$$U:= \bigl \{ u\in H_{\Omega}^{s}: \norm{u}_{L^{p+1}(\Omega)} = 1 \bigr \}.$$
We observe that $\bar{v}= \abs{\Omega}^{\frac{-1}{p+1}}\in U.$
\begin{rem}
$U$ is a $C^{1}$ manifold.
\end{rem}
\begin{rem}\label{rem4.4}
It is easy to observe that the minimizer $v_{\eta}$ for $X_{\lambda}(\Omega_{\eta})$ when normalized 
as $\norm{\bar{v}_{\eta}}_{L^{p+1}(\Omega)}^{2}=1$ satisfies the following problem in the weak sense:
\begin{equation}\label{eqn:P4}
  \begin{array} { l l } {(-\Delta)^{s}u+ \lambda \eta^{2s} u=\eta^{2s}X_{\lambda}(\Omega_{\eta})\eta^{-n\big(\frac{p-1}{p+1}\big)}
  \abs{u}^{p-1}u}, & {x \in \Omega}\\{\hspace{1cm} \mathcal{N}_{s}u(x)=0, } & { x \in \mathbb{R}^{n}\setminus \bar{\Omega}}. 
  \\ \end{array}
\end{equation} 
  \end{rem}

\noindent Now, we define the functional $$G: U \times [0,1) \rightarrow (H_{\Omega}^{s})^{*}$$ such that 
\begin{align}
G(v,\eta)(w):= \frac{c_{n,s}}{2} \int_{T(\Omega)}\frac{(v(x)-v(y))(w(x)-w(y)}{\abs{x-y}^{n+2s}}+ \lambda \eta^{2s}\int_{\Omega} vwdx-
\eta^{2s}X_{s}(\Omega_{\eta})\eta^{-n\big(\frac{p-1}{p+1}\big)}\int_{\Omega}\abs{v}^{p-1}vw dx.
\end{align}

\begin{rem}
We use the Implicit Function Theorem (IFT) to show the existence of a small number $\delta>0$ and a curve 
$\varphi: [0,\delta) \rightarrow U$ such that
\begin{enumerate}
\item $\varphi(0)=\bar{v}=\abs{\Omega}^{-\frac{1}{p+1}} 
\text{and}~ G(\varphi(\eta), \eta)=0 ~\text{for every}~ 0\leq \eta < \delta. $ \\
\item if $(v, \eta) \in U\times [0,\delta)$ is such that $G(v,\eta)=0$ and $v$ is near to $\bar{v}$ then $\varphi(\eta)= \bar{v}$.
\end{enumerate}
\end{rem}
In  order to apply IFT,  we show that the derivative of $G$ with respect to $v$ at point $(\bar{v}, 0),$
i.e.,  ${G_{v}|}_{(\bar{v},0)}$ is invertible, see for example \cite{A5,H1}. For this, we need the following results.\\

\noindent Following the similar lines of proof as in Lemma 4.1\,\cite{J2}, we have the following 
\begin{lem} 
The tangent space to U at $\bar{v}$ is given by 
 $$T_{\bar{v}}U=\biggl \{ u \in H_{\Omega}^{s}: \int_{\Omega}{u}dx=0 \biggr \}.$$
\end{lem}
For the sake of brevity, we omit the details.

\begin{lem}
The derivative ${G_{v}|}_{(\bar{v},0)}: T_{\bar{v}}U \rightarrow \mathcal{F}$ has a continuous inverse.
\end{lem}
\begin{proof} It is easy to see that $T_{\bar{v}}U$ is a Hilbert space. The inner product on this space is given by 
$$ \langle v,w \rangle = {G_{v}}_{|(\bar{v},0)}(v)(w)= \frac{c_{n,s}}{2} \int_{T(\Omega)}\frac{(v(x)-v(y))(w(x)-w(y)}{\abs{x-y}^{n+2s}} .$$
Then by an application of Riesz representation theorem, the conclusion follows. 
\end{proof}
Now, we are in a position to conclude the proof of Theorem\,\ref{minimizer4.6} by an application of IFT.\\

\noindent \textbf{Proof of Theorem \ref{minimizer4.6}}:   By IFT,  there exists a unique solution 
$v$ of $G(v,\eta)=0$ with $v$ close to $\bar{v}.$ Therefore, there exists a unique weak solution to \eqref{eqn:P4}
near $\bar{v}$ for small values of $\eta.$ By Lemma \ref{minimizer}, the minimizers converges 
to $\bar{v}$  in $H_{\Omega}^{s}$ as $\eta$ tends to zero. Also, from the above Remark \ref{rem4.4},  
these minimizers are weak solutions of problem (\ref{eqn:P4}). From this, the uniqueness of minimizers follows and this 
completes the proof. \qed \\
\section{Acknowledgement} 
The first author thanks CSIR for the financial support under the scheme 09/1031(0009)/2019-EMR-I.
The second author thanks DST/SERB
for the financial support under the grant CRG/2020/000041.
 
\end{document}